\documentclass[reqno]{amsart}
\usepackage{url}
\usepackage[dvips]{graphicx}

\newtheorem{thm}{Theorem}
\newtheorem{lem}{Lemma}
\newtheorem{defn}{Definition}
\begin{document}
\bibliographystyle{plain}
\title[Constructive proof of the existence of Nash Equilibrium]{Constructive proof of the existence of Nash Equilibrium in a finite strategic game with sequentially locally non-constant payoff functions by Sperner's lemma}
\author{Yasuhito Tanaka}
\address{Faculty of Economics, Doshisha University, Kamigyo-ku, Kyoto, 602-8580, Japan}
\email{yasuhito@mail.doshisha.ac.jp}

\date{}

\begin{abstract}
Using Sperner's lemma for modified partition of a simplex we will constructively prove the existence of a Nash equilibrium in a finite strategic game with sequentially locally non-constant payoff functions. We follow the Bishop style constructive mathematics.
\end{abstract}



\maketitle

\section{Introduction}
It is often said that Brouwer's fixed point theorem can not be constructively proved\footnote{\cite{kel} provided a \emph{constructive} proof of Brouwer's fixed point theorem. But it is not constructive from the view point of constructive mathematics \'{a} la Bishop. It is sufficient to say that one dimensional case of Brouwer's fixed point theorem, that is, the intermediate value theorem is non-constructive. See \cite{br} or \cite{da}. On the other hand, in \cite{orevkov} Orevkov constructed a computably coded continuous function $f$ from the unit square to itself, which is defined at each computable point of the square, such that $f$ has no computable fixed point. His map consists of a retract of the computable elements of the square to its boundary followed by a rotation of the boundary of the square. As pointed out by Hirst in \cite{hirst}, since there is no retract of the square to its boundary, Orevkov's map does not have a total extension.

Brouwer's fixed point theorem can be constructively, in the sense of constructive mathematics \'{a} la Bishop, proved only approximately. But the existence of an exact fixed point of a function which satisfies some property of local non-constancy may be constructively proved.}. Thus, the existence of a Nash equilibrium in a finite strategic game also can not be constructively proved. Sperner's lemma which is used to prove Brouwer's theorem, however, can be constructively proved. Some authors have presented a constructive (or an approximate) version of Brouwer's theorem using Sperner's lemma. See  \cite{da} and \cite{veld}. Thus, Brouwer's fixed point theorem can be constructively proved in its constructive version. Also Dalen in \cite{da} states a conjecture that a uniformly continuous function $f$ from a simplex to itself, with property that each open set contains a point $x$ such that $x\neq f(x)$, which means $|x-f(x)|>0$, and also at every point $x$ on the boundaries of the simplex $x\neq f(x)$, has an exact fixed point. We call such a property of functions \emph{local non-constancy}. Further we define a stronger property \emph{sequential local non-constancy}. In another paper \cite{ta1} we have constructively proved Dalen's conjecture with sequential local non-constancy. 

In this paper we present a proof of the existence of a Nash equilibrium in a finite strategic game with sequentially locally non-constant payoff functions by Sperner's lemma. We consider Sperner's lemma for modified partition of a simplex, and utilizing it prove the existence of such a Nash equilibrium.

In the next section we prove a modified version of Sperner's lemma. In Section 3 we present a proof of the existence of a Nash equilibrium in a finite strategic game with sequentially locally non-constant payoff functions by the modified version of Sperner's lemma. We follow the Bishop style constructive mathematics according to \cite{bb}, \cite{br} and \cite{bv}.

\section{Sperner's lemma}\label{sec2}
To prove Sperner's lemma we use the following simple result of graph theory, Handshaking lemma\footnote{For another constructive proof of Sperner's lemma, see \cite{su}. }. A \emph{graph} refers to a collection of vertices and a collection of edges that connect pairs of vertices. Each graph may be undirected or directed. Figure \ref{graph} is an example of an undirected graph. Degree of a vertex of a graph is defined to be the number of edges incident to the vertex, with loops counted twice. Each vertex has odd degree or even degree. Let $v$ denote a vertex and $V$ denote the set of all vertices.
\begin{lem}[Handshaking lemma]
Every undirected graph contains an even number of vertices of odd degree. That is, the number of vertices that have an odd number of incident edges must be even.
\end{lem}
This is a simple lemma. But for completeness of arguments we provide a proof.
\begin{proof}
Prove this lemma by double counting. Let $d(v)$ be the degree of vertex $v$. The number of vertex-edge incidences in the graph may be counted in two different ways; by summing the degrees of the vertices, or by counting two incidences for every edge. Therefore,
\[\sum_{v\in V}d(v)=2e,\]
where $e$ is the number of edges in the graph. The sum of the degrees of the vertices is therefore an even number. It could happen if and only if an even number of the vertices had odd degree.
\end{proof}

\begin{figure}[tpb]
\begin{center}
\includegraphics[height=5cm]{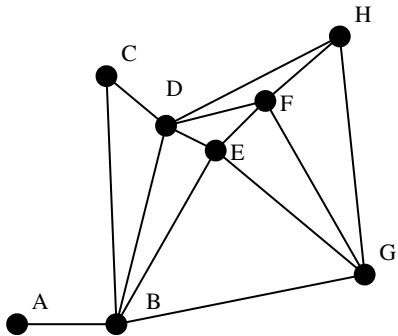}
\end{center}
	\vspace*{-.3cm}
	\caption{Example of graph}
	\label{graph}
\end{figure}

Let $\Delta$ denote an $n$-dimensional simplex. $n$ is a finite natural number. For example, a 2-dimensional simplex is a triangle.  Let partition or triangulate a simplex. Figure \ref{tria1} is an example of partition (triangulation) of a 2-dimensional simplex. In a 2-dimensional case we divide each side of $\Delta$ in $m$ equal segments, and draw the lines parallel to the sides of $\Delta$. Then, the 2-dimensional simplex is partitioned into $m^2$ triangles. We consider partition of $\Delta$ inductively for cases of higher dimension. In a 3 dimensional case each face of $\Delta$ is a 2-dimensional simplex, and so it is partitioned into $m^2$ triangles in the above mentioned way, and draw the planes parallel to the faces of $\Delta$. Then, the 3-dimensional simplex is partitioned into $m^3$ trigonal pyramids. And similarly for cases of higher dimension.

\begin{figure}[tpb]
\begin{center}
\includegraphics[height=7.5cm]{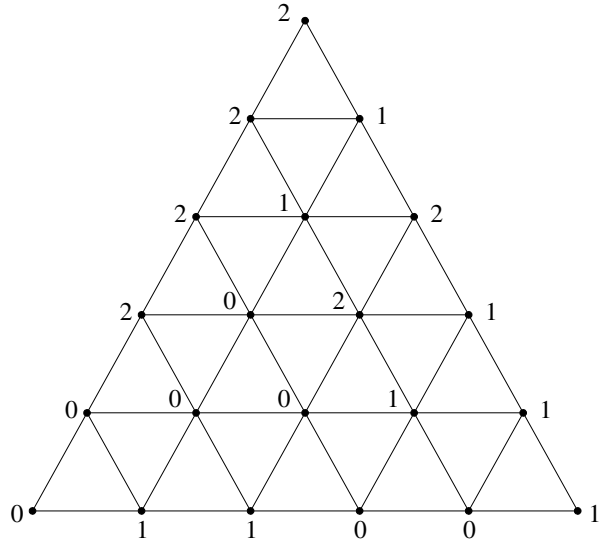}
\end{center}
	\vspace*{-.3cm}
	\caption{Partition and labeling of 2-dimensional simplex}
	\label{tria1}
\end{figure}

Let $K$ denote the set of small $n$-dimensional simplices of $\Delta$ constructed by partition. Vertices of these small simplices of $K$ are labeled with the numbers 0, 1, 2, $\dots$, $n$ subject to the following rules.
\begin{enumerate}
\item The vertices of $\Delta$ are respectively labeled with 0 to $n$. We label a point $(1,0, \dots, 0)$ with 0, a point $(0,1,0, \dots, 0)$ with 1, a point $(0,0,1 \dots, 0)$ with 2, $\dots$, a point $(0,\dots, 0,1)$ with $n$. That is, a vertex whose $k$-th coordinate ($k=0, 1, \dots, n$) is $1$ and all other coordinates are 0 is labeled with $k$ for all $k\in \{0, 1, \dots, n\}$. 

\item If a vertex of $K$ is contained in an $n-1$-dimensional face of $\Delta$, then this vertex is labeled with some number which is the same as the number of a vertex of that face.

\item If a vertex of $K$ is contained in an $n-2$-dimensional face of $\Delta$, then this vertex is labeled with some number which is the same as the number of a vertex of that face. And similarly for cases of lower dimension.

\item A vertex contained inside of $\Delta$ is labeled with an arbitrary number among 0, 1, $\dots$, $n$.
\end{enumerate}

Now we modify this partition of a simplex as follows;
\begin{quote}
Put a point in an open neighborhood around each vertex inside $\Delta$, and make partition of $\Delta$ replacing each vertex inside $\Delta$ by the point in each neighborhood. The diameter of each neighborhood should be sufficiently small relative to the size of each small simplex. We label the points in $\Delta$ following the rules (1) $\sim$ (4).
\end{quote}
Then we obtain a partition of $\Delta$ illustrated in Figure \ref{tria12}.

We further modify this partition as follows;
\begin{quote}
Put a point in an open neighborhood around each vertex on a face (boundary) of $\Delta$, and make partition of $\Delta$ replacing each vertex on the face by that point in each neighborhood, and we label the points in $\Delta$ following the rules (1) $\sim$ (4). This neighborhood is open in a space with dimension lower than $n$.
\end{quote}
Then, we obtain a partition of $\Delta$ depicted in Figure \ref{tria13}.

\begin{figure}[tpb]
\begin{center}
\includegraphics[height=7.5cm]{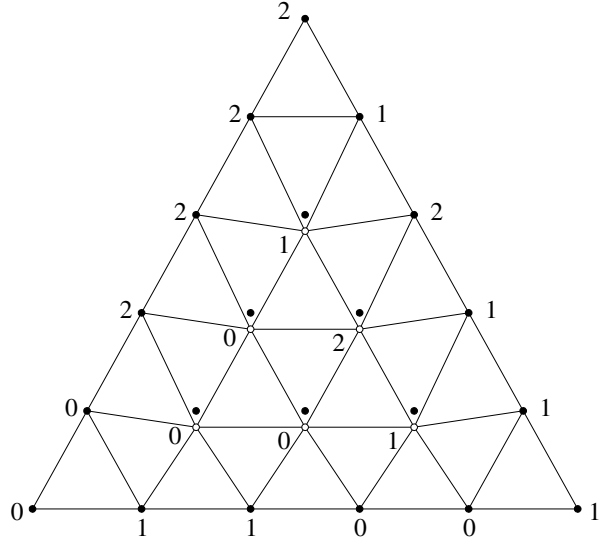}
\end{center}
	\vspace*{-.3cm}
	\caption{Modified partition of a simplex}
	\label{tria12}
\end{figure}

A small simplex of $K$ in this modified partition which is labeled with the numbers 0, 1, $\dots$, $n$ is called a \emph{fully labeled simplex}. Now let us prove Sperner's lemma about the modified partition of a simplex.
\begin{lem}[Sperner's lemma]
If we label the vertices of $K$ following above rules (1) $\sim$ (4), then there are an odd number of fully labeled simplices. Thus, there exists at least one fully labeled simplex. \label{l2}
\end{lem}
\begin{proof}
See Appendix \ref{app1}.
\end{proof}

\begin{figure}[tpb]
\begin{center}
\includegraphics[height=8cm]{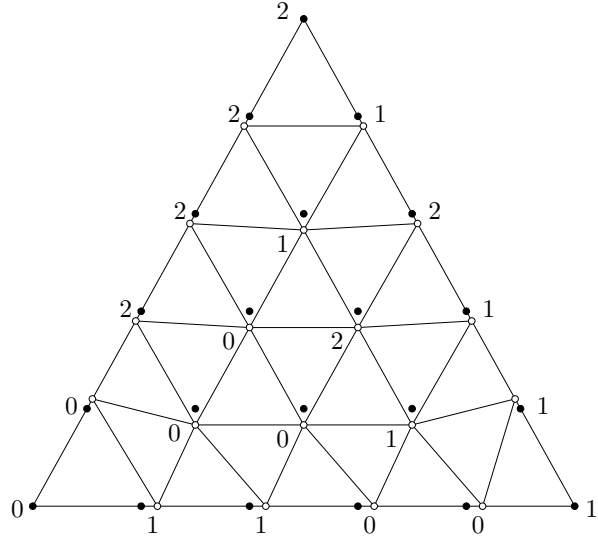}
\end{center}
	\vspace*{-.3cm}
	\caption{Modified partition of a simplex: Two}
	\label{tria13}
\end{figure}

\section{Nash equilibrium in strategic game}\label{sec4}

Let $\mathbf{p}=(\mathbf{p}_0, \mathbf{p}_1, \dots, \mathbf{p}_n)$ be a point in an $n$-dimensional simplex $\Delta$, and consider a function $\varphi$ from $\Delta$ to itself. Denote the $i$-th components of $\mathbf{p}$ and $\varphi(\mathbf{p})$ by $\mathbf{p}_i$ and $\varphi_i(\mathbf{p})$ or $\varphi_i$.

The definition of local non-constancy of functions is as follows;
\begin{defn}(\textit{Local non-constancy of functions})
\begin{enumerate}
	\item At a point $\mathbf{p}$ on the faces (boundaries) of a simplex $\varphi(\mathbf{p})\neq \mathbf{p}$. This means $\varphi_i(\mathbf{p})>\mathbf{p}_i$ or $\varphi_i(\mathbf{p})<\mathbf{p}_i$ for at least one $i$.
	\item In any open set in $\Delta$ there exists a point $\mathbf{p}$ such that $\varphi(\mathbf{p})\neq \mathbf{p}$.
\end{enumerate}
\end{defn}

Next we define modified local non-constancy of functions as follows;
\begin{defn}(\textit{Modified local non-constancy of functions})
\begin{enumerate}
\item At the vertices of a simplex $\Delta$ $\varphi(\mathbf{p})\neq \mathbf{p}$.
\item In any open set contained in the faces of $\Delta$ there exists a point $\mathbf{p}$ such that $\varphi(\mathbf{p})\neq \mathbf{p}$. This open set is open in a space with dimension lower than $n$.
	\item In any open set in $\Delta$ there exists a point $\mathbf{p}$ such that $\varphi(\mathbf{p})\neq \mathbf{p}$.
\end{enumerate}
\end{defn}

(2) of the modified local non-constancy implies that every vertex $\mathbf{p}$ in a partition of a simplex, for example, as illustrated by white circles in Figure \ref{tria13} in a 2-dimensional case, can be selected to satisfy $\varphi(\mathbf{p})\neq \mathbf{p}$ even when points on the faces of $\Delta$ (black circles on the edges) do not necessarily satisfy this condition. Even if a function $\varphi$ does not strictly satisfy the local non-constancy so long as it satisfies the modified local non-constancy, we can partition $\Delta$ to satisfy the conditions for Sperner's lemma.

Further, by reference to the notion of \emph{sequentially at most one maximum} in \cite{berg},  we define the property of \emph{sequential local non-constancy}.

First we recapitulate the compactness (total boundedness with completeness) of a set in constructive mathematics. $\Delta$ is compact in the sense that for each $\varepsilon>0$ there exists a finitely enumerable $\varepsilon$-approximation to $\Delta$\footnote{A set $S$ is finitely enumerable if there exist a natural number $N$ and a mapping of the set $\{1, 2, \dots, N\}$ onto $S$.}. An $\varepsilon$-approximation to $\Delta$ is a subset of $\Delta$ such that for each $\mathbf{p}\in \Delta$ there exists $\mathbf{q}$ in that $\varepsilon$-approximation with $|\mathbf{p}-\mathbf{q}|<\varepsilon$. Each face (boundary) of $\Delta$ is also a simplex, and so it is compact. According to Corollary 2.2.12 of \cite{bv} we have the following result.
\begin{lem}
For each $\varepsilon>0$ there exist totally bounded sets $H_1, H_2, \dots, H_n$, each of diameter less than or equal to $\varepsilon$, such that $\Delta=\cup_{i=1}^nH_i$.\label{closed}
\end{lem}

The definition of sequential local non-constancy is as follows;
\begin{defn}(\textit{Sequential local non-constancy of functions})
\begin{enumerate}
\item At the vertices of a simplex $\Delta$ $\varphi(\mathbf{p})\neq \mathbf{p}$.

\item Let $\partial \Delta$ be a face of $\Delta$. There exists $\bar{\varepsilon}$ with the following property. For each $\varepsilon>0$ less than $\bar{\varepsilon}$ there exist totally bounded sets $H_1, H_2, \dots, H_m$, each of diameter less than or equal to $\varepsilon$, such that $\partial \Delta=\cup_{i=1}^m H_i$, and if for all sequences $(\mathbf{p}_n)_{n\geq 1}$, $(\mathbf{q}_n)_{n\geq 1}$ in each $H_i$, $|\varphi(\mathbf{p}_n)-\mathbf{p}_n|\longrightarrow 0$ and $|\varphi(\mathbf{q}_n)-\mathbf{q}_n|\longrightarrow 0$, then $|\mathbf{p}_n-\mathbf{q}_n|\longrightarrow 0$.

\item There exists $\bar{\varepsilon}$ with the following property. For each $\varepsilon>0$ less than $\bar{\varepsilon}$ there exist totally bounded sets $H_1, H_2, \dots, H_m$, each of diameter less than or equal to $\varepsilon$, such that $\Delta=\cup_{i=1}^m H_i$, and if for all sequences $(\mathbf{p}_n)_{n\geq 1}$, $(\mathbf{q}_n)_{n\geq 1}$ in each $H_i$, $|\varphi(\mathbf{p}_n)-\mathbf{p}_n|\longrightarrow 0$ and $|\varphi(\mathbf{q}_n)-\mathbf{q}_n|\longrightarrow 0$, then $|\mathbf{p}_n-\mathbf{q}_n|\longrightarrow 0$.
\end{enumerate}\label{sln}
\end{defn}
(1) of this definition is the same as that of the definition of modified local non-constancy. 

Now we show the following two lemmas.
\begin{lem}
Sequential local non-constancy means modified local non-constancy.\label{seq}
\end{lem}
The essence of this proof is due to the proof of Proposition 1 of \cite{berg}.
\begin{proof}
Let $H_i$ be a set as defined above. Construct a sequence $(\mathbf{r}_n)_{n\geq 1}$ in $H_i$ such that $|\varphi(\mathbf{r}_n)-\mathbf{r}_n|\longrightarrow 0$. Consider $\mathbf{p}$, $\mathbf{q}$ in $H_i$ with $\mathbf{p}\neq \mathbf{q}$. Construct an increasing binary sequence $(\lambda_n)_{n\geq 1}$ such that
\[\lambda_n=0\Rightarrow \max(|\varphi(\mathbf{p})-\mathbf{p}|, |\varphi(\mathbf{q})-\mathbf{q}|)<2^{-n},\]
\[\lambda_n=1\Rightarrow \max(|\varphi(\mathbf{p})-\mathbf{p}|, |\varphi(\mathbf{q})-\mathbf{q}|)>2^{-n-1}.\]
We may assume that $\lambda_1=0$. If $\lambda_n=0$, set $\mathbf{p}_n=\mathbf{p}$ and $\mathbf{q}_n=\mathbf{q}$. If $\lambda_n=1$, set $\mathbf{p}_n=\mathbf{q}_n=\mathbf{r}_n$. Now the sequences $(|\varphi(\mathbf{p}_n)-\mathbf{p}_n|)_{n\geq 1}$, $(|\varphi(\mathbf{q}_n)-\mathbf{q}_n|)_{n\geq 1}$ converge to 0, and so $|\mathbf{p}_n-\mathbf{q}_n|\longrightarrow 0$. Computing $N$ such that $|\mathbf{p}_N-\mathbf{q}_N|<|\mathbf{p}-\mathbf{q}|$, we see that $\lambda_N=1$. Therefore, $\varphi(\mathbf{p})\neq \mathbf{p}$ or $\varphi(\mathbf{q})\neq \mathbf{q}$.
\end{proof}
The converse of Lemma \ref{seq} does not hold because the sequential local non-constancy implies isolatedness of points $\mathbf{r}$ satisfying $\varphi(\mathbf{r})=\mathbf{r}$ but the local non-constancy and the modified local non-constancy do not.

\begin{lem}
Let $\varphi$ be a uniformly continuous function from $\Delta$ to itself. Assume $\inf_{\mathbf{p}\in H_i}\varphi(\mathbf{p})=0$ for $H_i\subset \Delta$ defined above. If the following property holds:
\begin{quote}
For each $\delta>0$ there exists $\varepsilon>0$ such that if $\mathbf{p}, \mathbf{q}\in H_i$, $|\varphi(\mathbf{p})-\mathbf{p}|<\varepsilon$ and $|\varphi(\mathbf{q})-\mathbf{q}|<\varepsilon$, then $|\mathbf{p}-\mathbf{q}|\leq \delta$.
\end{quote}
Then, there exists a point $\mathbf{r}\in H_i$ such that $\varphi(\mathbf{r})=\mathbf{r}$. \label{fix0}
\end{lem}
\begin{proof}
Choose a sequence $(\mathbf{p}_n)_{n\geq 1}$ in $H_i$ such that $|\varphi(\mathbf{p}_n)-\mathbf{p}_n|\longrightarrow 0$. Compute $N$ such that $|\varphi(\mathbf{p}_n)-\mathbf{p}_n|<\varepsilon$ for all $n\geq N$. Then, for $m, n\geq N$ we have $|\mathbf{p}_m-\mathbf{p}_n|\leq \delta$. Since $\delta>0$ is arbitrary, $(\mathbf{p}_n)_{n\geq 1}$ is a Cauchy sequence in $H_i$, and converges to a limit $\mathbf{r}\in H_i$. The continuity of $\varphi$ yields $|\varphi(\mathbf{r})-\mathbf{r}|=0$, that is, $\varphi(\mathbf{r})=\mathbf{r}$.
\end{proof}

Now we look at the problem of the existence of a Nash equilibrium in a finite strategic game according to \cite{nash}.  A Nash equilibrium of a finite strategic game is a state where all players choose their best responses to strategies of other players.

Consider an $n$-players strategic game with $m$ pure strategies for each player. $n$ and $m$ are finite natural numbers not smaller than 2. Let $S_i$ be the set of pure strategies of player $i$, and denote his each pure strategy by $s_{ij}$. His mixed strategy is defined as a probability distribution over $S_i$, and is denoted by $\mathbf{p}_i$. Let $p_{ij}$ be a probability that player $i$ chooses $s_{ij}$, then we must have $\sum^m_{j=1}p_{ij}=1$ for all $i$. A combination of mixed strategies of all players is called a \emph{profile}. It is denoted by $\mathbf{p}$. Let $\pi_i(\mathbf{p})$ be the expected payoff of player $i$ at profile $\mathbf{p}$, and $\pi_i(s_{ij},\mathbf{p}_{-i})$ be his payoff when he chooses a strategy $s_{ij}$ at that profile, where $\mathbf{p}_{-i}$ denotes a combination of mixed strategies of players other than $i$ at profile $\mathbf{p}$. $\pi_i(\mathbf{p})$ is written as follows;
\[\pi_i(\mathbf{p})=\pi_i(\mathbf{p}_i, \mathbf{p}_{-i})=\sum_{\{j: p_{ij}>0\}}p_{ij}\pi_i(s_{ij}, \mathbf{p}_{-i})\]
Assume that the values of payoffs of all players are finite. Then, since pure strategies are finite, and expected payoffs are linear functions about probability distributions over the sets of pure strategies of all players, $\pi_i(\mathbf{p})$ is uniformly continuous about $\mathbf{p}$.

For each $i$ and $j$ let
\[v_{ij}=p_{ij}+\max(\pi_i(s_{ij}, \mathbf{p}_{-i})-\pi_i(\mathbf{p}),0),\]
and define the following function.
\begin{equation}
\psi_{ij}(\mathbf{p})=\frac{v_{ij}}{v_{i1}+v_{i2}+\dots +v_{im}},\label{psi1}
\end{equation}
where $\sum_{j=1}^m\psi_{ij}=1$ for all $i$. Let $\psi_i(\mathbf{p})=(\psi_{i1}, \psi_{i2}, \dots, \psi_{im})$, $\psi(\mathbf{p})=(\psi_1, \psi_2, \cdots, \psi_n)$. Since each $\psi_i$ is an $m$-dimensional vector such that the values of its components are between 0 and 1, and the sum of its components is 1, it represents a point on an $m-1$-dimensional simplex. $\psi(\mathbf{p})$ is a combination of vectors $\psi_i$'s. It is a vector such that its components are components of $\psi_i(\mathbf{p})$ for all players. Thus, it is a vector with $n\times m$ components, but since the number of independent components is $n(m-1)$, the range of $\psi$ is the $n$-times product of $m-1$-dimensional simplices. It is convex, and homeomorphic to an $n(m-1)$-dimensional simplex. $\mathbf{p}=(\mathbf{p}_1, \mathbf{p}_2, \dots, \mathbf{p}_n)$ is also a vector with $n\times m$ components, and the number of its independent components is $n(m-1)$. Thus, the domain of $\psi$ is the same set as the range of $\psi$, and a uniformly continuous function from the domain of $\psi$ to its range corresponds one to one to a uniformly continuous function from an $n(m-1)$-dimensional simplex to itself.

\begin{figure}[tpb]
\begin{center}
\includegraphics[height=7.5cm]{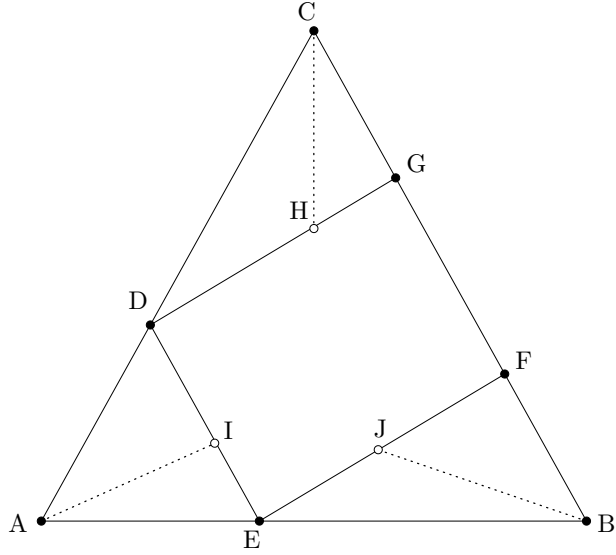}

\end{center}
	\vspace*{-.3cm}
	\caption{Homeomorphism between simplex and combination of strategies}
	\label{tria15}
\end{figure}
%}

Let us consider a homeomorphism between an $n(m-1)$-dimensional simplex and the space of players' mixed strategies which is denoted by $\mathbf{P}$. Figure \ref{tria15} depicts an example of a case of two players with two pure strategies for each player. Vertices $D$, $E$, $F$ and $G$ represent states where two players choose pure strategies, and points on edges $DE$, $EF$, $FG$ and $GD$ represent states where one player chooses a pure strategy. In such a homeomorphism vertices of the simplex do not correspond to any vertex of $\mathbf{P}$. Vertices of the simplex and points on faces (simplices whose dimension is lower than $n(m-1)$) of the simplex correspond to the points on faces of $\mathbf{P}$. For example, in Figure \ref{tria15} $A$, $B$ and $C$ correspond, respectively, to $I$, $J$ and $H$. We make these points satisfy the condition of the following sequential local non-constancy. On the other hand, each vertex of $\mathbf{P}$, $D$, $E$, $F$ and $G$ corresponds, respectively, to itself on a face of the simplex which contains it. 
%For example, if $\mathbf{p}$ is $K$ in Figure \ref{tria15}, $\mathbf{p'}$ is a point in the face of the simplex which contains $K$, and it is in the near distance of $K$.
Next we assume the following conditions.
\begin{defn}(\textit{Sequential local non-constancy of payoff functions})
\begin{enumerate}
\item At points in $\mathbf{P}$ corresponding to vertices of $\Delta$ $\pi_i(s_{ij}, \mathbf{p}_{-i})-\pi_i(\mathbf{p})>0$ for some $i$ and $j$.

\item Let $\partial \mathbf{P}$ be a face of $\mathbf{P}$ in which some players choose pure strategies. There exists $\bar{\varepsilon}$ with the following property. For each $\varepsilon>0$ less than $\bar{\varepsilon}$ there exist totally bounded sets $H_1, H_2, \dots, H_m$, each of diameter less than or equal to $\varepsilon$, such that $\partial \mathbf{P}=\cup_{i=1}^m H_i$, and if for all sequences $(\mathbf{p}_n)_{n\geq 1}$, $(\mathbf{q}_n)_{n\geq 1}$ in each $H_i$, $\max(\pi_i(s_{ij}, (\mathbf{p}_n)_{-i})-\pi_i(\mathbf{p}_n), 0)\longrightarrow 0$, $\max(\pi_i(s_{ij}, (\mathbf{q}_n)_{-i})-\pi_i(\mathbf{q}_n), 0)\longrightarrow 0$ for all $s_{ij}\in S_i$ for all $i$, then $|\mathbf{p}_n-\mathbf{q}_n|\longrightarrow 0$.

\item There exists $\bar{\varepsilon}$ with the following property. For each $\varepsilon>0$ less than $\bar{\varepsilon}$ there exist totally bounded sets $H_1, H_2, \dots, H_m$, each of diameter less than or equal to $\varepsilon$, such that $\mathbf{P}=\cup_{i=1}^m H_i$, and if for all sequences $(\mathbf{p}_n)_{n\geq 1}$, $(\mathbf{q}_n)_{n\geq 1}$ in each $H_i$, $\max(\pi_i(s_{ij}, (\mathbf{p}_n)_{-i})-\pi_i(\mathbf{p}_n), 0)\longrightarrow 0$, $\max(\pi_i(s_{ij}, (\mathbf{q}_n)_{-i})-\pi_i(\mathbf{q}_n), 0)\longrightarrow 0$ for all $s_{ij}\in S_i$ for all $i$, then $|\mathbf{p}_n-\mathbf{q}_n|\longrightarrow 0$.
\end{enumerate}
 \label{sln2}
\end{defn}

By the sequential local non-constancy of payoff functions we obtain the following results.
\begin{enumerate}
	\item At points in $\mathbf{P}$ corresponding to vertices of $\Delta$, $\sum_{j, p_{ij}>0}p_{ij}\pi_i(s_{ij},\mathbf{p}_{-i})=\pi_i(\mathbf{p})$ for each $i$ implies that $\psi_{ij}(\mathbf{p})\neq p_{ij}$ for some $i$ and $j$.

	\item By (2) of the definition of sequential local non-constancy of payoff functions, for each $\varepsilon>0$ less than $\bar{\varepsilon}$ there exist totally bounded sets $H_1, H_2, \dots, H_m$, each of diameter less than or equal to $\varepsilon$, such that $\partial \mathbf{P}=\cup_{i=1}^m H_i$, and if for all sequences $(\mathbf{p}_n)_{n\geq 1}$, $(\mathbf{q}_n)_{n\geq 1}$ in each $H_i$, $|\psi(\mathbf{p}_n)-\mathbf{p}_n|\longrightarrow 0$ and $|\psi(\mathbf{q}_n)-\mathbf{q}_n|\longrightarrow 0$, then $|\mathbf{p}_n-\mathbf{q}_n|\longrightarrow 0$.

	\item By (3) of the definition of sequential local non-constancy of payoff functions, For each $\varepsilon>0$ less than $\bar{\varepsilon}$ there exist totally bounded sets $H_1, H_2, \dots, H_m$, each of diameter less than or equal to $\varepsilon$, such that $\mathbf{P}=\cup_{i=1}^m H_i$, and if for all sequences $(\mathbf{p}_n)_{n\geq 1}$, $(\mathbf{q}_n)_{n\geq 1}$ in each $H_i$, $|\psi(\mathbf{p}_n)-\mathbf{p}_n|\longrightarrow 0$ and $|\psi(\mathbf{q}_n)-\mathbf{q}_n|\longrightarrow 0$, then $|\mathbf{p}_n-\mathbf{q}_n|\longrightarrow 0$.
\end{enumerate}

Let replace $n$ by $n(m-1)$. We show the following theorem.
\begin{thm}
In any finite strategic game with sequentially locally non-constant payoff functions there exists a Nash equilibrium.
\end{thm}
\begin{proof}
Let us prove this theorem through some steps.
\begin{enumerate}
\item First we show that we can partition an $n(m-1)$-dimensional simplex $\Delta$ so that the conditions for Sperner's lemma (for modified partition of a simplex) are satisfied. We partition $\Delta$ according to the method in the proof of Sperner's lemma, and label the vertices of simplices constructed by partition of $\Delta$. It is important how to label the vertices contained in the faces of $\Delta$. Let $K$ be the set of small simplices constructed by partition of $\Delta$, $\mathbf{p}=(\mathbf{p}_0, \mathbf{p}_1, \dots, \mathbf{p}_{n(m-1)})$ be a vertex of a simplex of $K$, and denote the $i$-th coordinate of $\psi(\mathbf{p})$ by $\psi_i$ or $\psi_i(\mathbf{p})$. We label a vertex $\mathbf{p}$ according to the following rule,
\[\mathrm{If}\ \mathbf{p}_k>\psi_k,\ \mathrm{we\ label}\ \mathbf{p}\ \mathrm{with}\ k.\]
If there are multiple $k$'s which satisfy this condition, we label $\mathbf{p}$ conveniently for the conditions for Sperner's lemma to be satisfied.% We do not randomly label the vertices.

We consider labeling for vertices about three cases.
\begin{enumerate}
	\item Vertices of $\Delta$:

One of the coordinates of a vertex $\mathbf{p}$ of $\Delta$ is 1, and all other coordinates are zero. Consider a vertex $(1, 0, \dots, 0)$. $\psi(\mathbf{p})\neq \mathbf{p}$ means $\psi_j>\mathbf{p}_j$ or $\psi_j<\mathbf{p}_j$ for at least one $j$. $\psi_i<\mathbf{p}_i$ can not hold for $i\neq 0$. On the other hand, $\psi_0>\mathbf{p}_0$ can not hold. When $\psi_0(\mathbf{p})<\mathbf{p}_0$, we label $\mathbf{p}$ with 0. Assume that $\psi_i(\mathbf{p})>\mathbf{p}_i=0$ for some $i\neq 0$. Then, since $\sum_{j=0}^{n(m-1)}\psi_j(\mathbf{p})=1=\mathbf{p}_0$, we have $\psi_0(\mathbf{p})<\mathbf{p}_0$. Therefore, $\mathbf{p}$ is labeled with 0. Similarly a vertex $\mathbf{p}$ whose $k$-th coordinate is 1 is labeled with $k$ for all $k\in \{0, 1, \dots, n(m-1)\}$.

	\item Vertices on the faces of $\Delta$:

Let $\mathbf{p}$ be a vertex of a simplex contained in an $n(m-1)-1$-dimensional face of $\Delta$ such that $\mathbf{p}_i=0$ for one $i$ among $0, 1, 2, \dots, n(m-1)$ (its $i$-th coordinate is 0). $\psi(\mathbf{p})\neq \mathbf{p}$ means that $\psi_j>\mathbf{p}_j$ or $\psi_j<\mathbf{p}_j$ for at least one $j$. $\psi_i<\mathbf{p}_i=0$ can not hold. When $\psi_k<\mathbf{p}_k$ for some $k\neq i$, we label $\mathbf{p}$ with $k$. Assume $\psi_i>\mathbf{p}_i=0$. Then, since $\sum_{j=0}^{n(m-1)}\mathbf{p}_j=\sum_{j=0}^{n(m-1)}\psi_j=1$, we have $\psi_k<\mathbf{p}_k$ for some $k\neq i$, and we label $\mathbf{p}$ with $k$. Assume that $\psi_j>\mathbf{p}_j$ for some $j\neq i$. Then, we have
\[\sum_{l=0, l\neq j}^{n(m-1)} \psi_l<\sum_{l=0, l\neq j}^{n(m-1)}\mathbf{p}_l,\]
and so $\psi_k<\mathbf{p}_k$ for some $k\neq i, j$. Thus, we label $\mathbf{p}$ with $k$.

We have proved that we can label each vertex of a simplex contained in an $n(m-1)-1$-dimensional face of $\Delta$ such that $\mathbf{p}_i=0$ for one $i$ among $0, 1, 2, \dots, n(m-1)$ with the number other than $i$. By similar procedures we can show that we can label the vertices of a simplex contained in an $n(m-1)-2$-dimensional face of $\Delta$ such that $\mathbf{p}_i=0$ for two $i$'s among $0, 1, 2, \dots, n(m-1)$ with the number other than those $i$'s, and so on.

\begin{quote}
Consider the case where, for example, $\mathbf{p}_{i}=\mathbf{p}_{i+1}=0$. Neither $\psi_i<\mathbf{p}_i=0$ nor $\psi_i<\mathbf{p}_{i+1}=0$ can hold. When $\psi_k<\mathbf{p}_k$ for some $j\neq i, i+1$, we label $\mathbf{p}$ with $k$. Assume $\psi_i>\mathbf{p}_i=0$ or $\psi_{i+1}>\mathbf{p}_{i+1}=0$. Then, since $\sum_{j=0}^{n(m-1)}\mathbf{p}_j=\sum_{j=0}^{n(m-1)}\psi_j=1$, we have $\psi_k<\mathbf{p}_k$ for some $k\neq i,\ i+1$, and we label $\mathbf{p}$ with $k$. Assume that $\psi_j>\mathbf{p}_j$ for some $j\neq i, i+1$. Then, we have
\[\sum_{l=0, l\neq j}^{n(m-1)} \psi_l<\sum_{l=0, l\neq j}^{n(m-1)}\mathbf{p}_l,\]
and so $\psi_k<\mathbf{p}_k$ for some $k\neq i, i+1, j$. Thus, we label $\mathbf{p}$ with $k$.
\end{quote}

	\item Vertices of small simplices inside $\Delta$:

By the modified local non-constancy of $\psi$ every vertex $\mathbf{p}$ in a modified partition of a simplex can be selected to satisfy $\psi(\mathbf{p})\neq \mathbf{p}$. Assume that $\psi_i>\mathbf{p}_i$ for some $i$. Then, since $\sum_{j=0}^{n(m-1)} \mathbf{p}_j=\sum_{j=0}^{n(m-1)} \psi_j=1$, we have
\[\psi_k<\mathbf{p}_k\]
for some $k\neq i$, and we label $\mathbf{p}$ with $k$.
\end{enumerate}
Therefore, the conditions for Sperner's lemma (for modified partition of a simplex) are satisfied, and there exist an odd number of fully labeled simplices in $K$.

\item Suppose that we partition $\Delta$ sufficiently fine so that the distance between any pair of the vertices of simplices of $K$ is sufficiently small. Let $\delta^{n(m-1)}$ be a fully labeled $n(m-1)$-dimensional simplex of $K$, and $\mathbf{p}^0, \mathbf{p}^1, \dots$ and $\mathbf{p}^{n(m-1)}$ be the vertices of $\delta^{n(m-1)}$. We name these vertices so that $\mathbf{p}^0, \mathbf{p}^1, \dots, \mathbf{p}^{n(m-1)}$ are labeled, respectively, with 0, 1, $\dots$, $n(m-1)$. The values of $\psi$ at theses vertices are $\psi(\mathbf{p}^0), \psi(\mathbf{p}^1), \dots$ and $\psi(\mathbf{p}^{n(m-1)})$. The $j$-th coordinates of $\mathbf{p}^i$ and $\psi({\mathbf{p}^i}),\ i=0, 1, \dots, n(m-1)$, are, respectively, denoted by $\mathbf{p}^i_j$ and $\psi_j(\mathbf{p}^i)$. About $\mathbf{p}^0$, from the labeling rules we have $\mathbf{p}^0_0>\psi_0(\mathbf{p}^0)$. About $\mathbf{p}^1$, also from the labeling rules we have $\mathbf{p}^1_1>\psi_1(\mathbf{p}^1)$. Since $n$ and $m$ are finite, by the uniform continuity of $\psi$ there exists $\delta>0$ such that if $|\mathbf{p}^i-\mathbf{p}^j|<\delta$, then $|\psi(\mathbf{p}^i)-\psi(\mathbf{p}^j)|<\frac{\varepsilon}{2n(m-1)[n(m-1)+1]}$ for $\varepsilon>0$ and $i\neq j$. $|\psi(\mathbf{p}^0)-\psi(\mathbf{p}^1)|<\frac{\varepsilon}{2n(m-1)[n(m-1)+1]}$ means $\psi_1(\mathbf{p}^1)>\psi_1(\mathbf{p}^0)-\frac{\varepsilon}{2n(m-1)[n(m-1)+1]}$. On the other hand, $|\mathbf{p}^0-\mathbf{p}^1|<\delta$ means $\mathbf{p}^0_1>\mathbf{p}^1_1-\delta$. We can make $\delta$ satisfying $\delta<\frac{\varepsilon}{2n(m-1)[n(m-1)+1]}$. Thus, from
\begin{align*}
\mathbf{p}^0_1&>\mathbf{p}^1_1-\delta,\ \mathbf{p}^1_1>\psi_1(\mathbf{p}^1),\ \psi_1(\mathbf{p}^1)\\
&>\psi_1(\mathbf{p}^0)-\frac{\varepsilon}{2n(m-1)[n(m-1)+1]}
\end{align*}
we obtain
\begin{align*}
\mathbf{p}^0_1&>\psi_1(\mathbf{p}^0)-\delta-\frac{\varepsilon}{2n(m-1)[n(m-1)+1]}\\
&>\psi_1(\mathbf{p}^0)-\frac{\varepsilon}{n(m-1)[n(m-1)+1]}
\end{align*}
By similar arguments, for each $i$ other than 0,
\begin{equation}
\mathbf{p}^0_i>\psi_i(\mathbf{p}^0)-\frac{\varepsilon}{n(m-1)[n(m-1)+1]}. \label{e1}
\end{equation}
For $i=0$ we have
\begin{equation}
\mathbf{p}^0_0>\psi_0(\mathbf{p}^0) \label{e2}
\end{equation}
Adding (\ref{e1}) and (\ref{e2}) side by side except for some $i$ (denote it by $k$) other than 0,
\[\sum_{j=0, j\neq k}^{n(m-1)} \mathbf{p}^0_j>\sum_{j=0, j\neq k}^{n(m-1)} \psi_j(\mathbf{p}^0)-\frac{[n(m-1)-1]\varepsilon}{n(m-1)[n(m-1)+1]}.\]
From $\sum_{j=0}^{n(m-1)} \mathbf{p}^0_j=1$, $\sum_{j=0}^{n(m-1)} \psi_j(\mathbf{p}^0)=1$ we have $1-\mathbf{p}^0_k>1-\psi_k(\mathbf{p}^0)-\frac{[n(m-1)-1]\varepsilon}{n(m-1)[n(m-1)+1]}$, which is rewritten as
\[\mathbf{p}^0_k<\psi_k(\mathbf{p}^0)+\frac{[n(m-1)-1]\varepsilon}{n(m-1)[n(m-1)+1]}.\]
Since (\ref{e1}) implies $\mathbf{p}^0_k>\psi_k(\mathbf{p}^0)-\frac{\varepsilon}{n(m-1)[n(m-1)+1]}$, we have
\begin{align*}
\psi_k(\mathbf{p}^0)&-\frac{\varepsilon}{n(m-1)[n(m-1)+1]}<\mathbf{p}^0_k<\psi_k(\mathbf{p}^0)\\
&+\frac{[n(m-1)-1]\varepsilon}{n(m-1)[n(m-1)+1]}.
\end{align*}
Thus,
\begin{equation}
|\mathbf{p}^0_k-\psi_k(\mathbf{p}^0)|<\frac{[n(m-1)-1]\varepsilon}{n(m-1)[n(m-1)+1]} \label{e3}
\end{equation}
is derived. On the other hand, adding (\ref{e1}) from 1 to $n(m-1)$ yields
\begin{equation*}
\sum_{j=1}^{n(m-1)} \mathbf{p}^0_j>\sum_{j=1}^{n(m-1)} \psi_j(\mathbf{p}^0)-\frac{\varepsilon}{n(m-1)+1}.%\label{nash1}
\end{equation*}
From $\sum_{j=0}^{n(m-1)} \mathbf{p}^0_j=1$, $\sum_{j=0}^{n(m-1)} \psi_j(\mathbf{p}^0)=1$ we have
\begin{equation}
1-\mathbf{p}^0_0>1-\psi_0(\mathbf{p}^0)-\frac{\varepsilon}{n(m-1)+1}.\label{nash1}
\end{equation}
Then, from (\ref{e2}) and (\ref{nash1}) we get
\begin{equation}
|\mathbf{p}^0_0-\psi_0(\mathbf{p}^0)|<\frac{\varepsilon}{n(m-1)+1}. \label{e24}
\end{equation}
From (\ref{e3}) and (\ref{e24}) we obtain the following result,
\begin{equation*}
|\mathbf{p}^0_i-\psi_i(\mathbf{p}^0)|<\frac{\varepsilon}{n(m-1)+1}\ \mathrm{for\ all}\ i. %\label{fp}
\end{equation*}
Thus,
\begin{equation}
|\mathbf{p}^0-\psi(\mathbf{p}^0)|<\varepsilon.\label{fp}
\end{equation}
Since $\varepsilon$ is arbitrary, we have $\inf_{\mathbf{p}\in \Delta}|\psi(\mathbf{p})-\mathbf{p}|=0$.

\item Since as described in Definition \ref{sln} $\Delta=\sum_{i=1}^nH_i$, where each $H_i$ is a totally bounded set whose diameter is less than or equal to $\varepsilon$, for at least one $H_i$ we have $\inf_{\mathbf{p}\in H_i}|\psi(\mathbf{p})-\mathbf{p}|=0$. Choose a sequence $(\mathbf{r}_n)_{n\geq 1}$ such that $|\psi(\mathbf{r}_n)-\mathbf{r}_n|\longrightarrow 0$ in such $H_i$. In view of Lemma \ref{fix0} it is enough to prove that the following condition holds.
\begin{quote}
For each $\varepsilon>0$ there exists $\delta>0$ such that if $\mathbf{p}, \mathbf{q}\in H_i$, $|\psi(\mathbf{p})-\mathbf{p}|<\delta$ and $|\psi(\mathbf{q})-\mathbf{q}|<\delta$, then $|\mathbf{p}-\mathbf{q}|\leq \varepsilon$.
\end{quote}
Assume that the set
\[T=\{(\mathbf{p},\mathbf{q})\in H_i\times H_i:\ |\mathbf{p}-\mathbf{q}|\geq \varepsilon\}\]
is nonempty and compact\footnote{See Theorem 2.2.13 of \cite{bv}.}. Since the mapping $(\mathbf{p},\mathbf{q})\longrightarrow \max(|\psi(\mathbf{p})-\mathbf{p}|,|\psi(\mathbf{q})-\mathbf{q}|)$ is uniformly continuous, we can construct an increasing binary sequence $(\lambda_n)_{n\geq 1}$ such that
\[\lambda_n=0\Rightarrow \inf_{(\mathbf{p},\mathbf{q})\in T}\max(|\psi(\mathbf{p})-\mathbf{p}|,|\psi(\mathbf{q})-\mathbf{q}|)<2^{-m},\]
\[\lambda_n=1\Rightarrow \inf_{(\mathbf{p},\mathbf{q})\in T}\max(|\psi(\mathbf{p})-\mathbf{p}|,|\psi(\mathbf{q})-\mathbf{q}|)>2^{-m-1}.\]
It suffices to find $n$ such that $\lambda_n=1$. In that case, if $|\psi(\mathbf{p})-\mathbf{p}|<2^{-m-1}$, $|\psi(\mathbf{q})-\mathbf{q}|<2^{-m-1}$, we have $(\mathbf{p},\mathbf{q})\notin T$ and $|\mathbf{p}-\mathbf{q}|\leq \varepsilon$. Assume $\lambda_1=0$. If $\lambda_n=0$, choose $(\mathbf{p}_n, \mathbf{q}_n)\in T$ such that $\max(|\psi(\mathbf{p}_n)-\mathbf{p}_n|, |\psi(\mathbf{q}_n)-\mathbf{q}_n|)<2^{-m}$, and if $\lambda_n=1$, set $\mathbf{p}_n=\mathbf{q}_n=\mathbf{r}_n$. Then, $|\psi(\mathbf{p}_n)-\mathbf{p}_n|\longrightarrow 0$ and $|\psi(\mathbf{q}_n)-\mathbf{q}_n|\longrightarrow 0$, so $|\mathbf{p}_n-\mathbf{q}_n|\longrightarrow 0$. Computing $N$ such that $|\mathbf{p}_N-\mathbf{q}_N|<\varepsilon$, we must have $\lambda_N=1$. We have completed the proof of the existence of a point which satisfies $\psi(\mathbf{p})=\mathbf{p}$.

\item Denote one of the points which satisfy $\psi(\mathbf{p})=\mathbf{p}$ by $\mathbf{\tilde{p}}=(\tilde{p}_1, \tilde{p}_2, \dots, \tilde{p}_n)$ and the components of $\tilde{p}_i$ by $\tilde{p}_{ij}$. Then, we have
\[\psi_{ij}=\tilde{p}_{ij},\ \mathrm{for\ all}\ i\ \mathrm{and}\ j.\]
By the definition of $\psi_{ij}$

\begin{align*}
\frac{\tilde{p}_{ij}+\max(\pi_i(s_{ij}, \mathbf{\tilde{p}}_{-i})-\pi_{i}(\mathbf{\tilde{p}}),0)}{1+\sum_{k=1}^m\max(\pi_i(s_{ik}, \mathbf{\tilde{p}}_{-i})-\pi_i(\mathbf{\tilde{p}}),0)}=\tilde{p}_{ij}.
\end{align*}
Let $\lambda=\sum_{k=1}^m\max(\pi_i(s_{ik}, \mathbf{\tilde{p}}_{-i})-\pi_i(\mathbf{\tilde{p}}),0)$, then
\[\max(\pi_i(s_{ij}, \mathbf{\tilde{p}}_{-i})-\pi_i(\mathbf{\tilde{p}}),0)=\lambda \tilde{p}_{ij},\]
where $\mathbf{\tilde{p}}_{-i}$ denotes a combination of mixed strategies of players other than $i$ at profile $\mathbf{\tilde{p}}$.

Since $\pi_i(\mathbf{\tilde{p}})=\sum_{\{j: \tilde{p}_{ij}>0\}}\tilde{p}_{ij}\pi_i(s_{ij}, \mathbf{\tilde{p}}_{i})$, it is impossible that $\max(\pi_i(s_{ij}, \mathbf{\tilde{p}}_{-i})-\pi_i(\mathbf{\tilde{p}}),0)=\pi_i(s_{ij}, \mathbf{\tilde{p}}_{-i})-\pi_i(\mathbf{\tilde{p}})>0$ for all $j$ satisfying $\tilde{p}_{ij}>0$. Thus, $\lambda=0$, and $\max(\pi_i(s_{ij}, \mathbf{\tilde{p}}_{-i})-\pi_i(\mathbf{\tilde{p}}),0)=0$ holds for all $s_{ij}$'s whether $\tilde{p}_{ij}>0$ or not, and it holds for all players. Then, strategies of all players in $\mathbf{\tilde{p}}$ are best responses each other, and a state where all players choose these strategies is a Nash equilibrium.

\end{enumerate}
\end{proof}

\begin{table}[htpb]
\begin{center}
\begin{tabular}{c|c|c|c|}
\multicolumn{4}{c}{\ \ \ \ \ \ \ \ \ \ \ \ \ \ \ \ Player 2} \\ \cline{2-4}
& &X&Y\\ \cline{2-4}
Player&X&2, 2 &0, 3\\ \cline{2-4}
1&Y& 3, 0& 1, 1\\ \cline{2-4}
\end{tabular}
\vspace{.3cm}
\caption{Example of game 1}
\label{ga-1}
\end{center}
\end{table}

Consider two examples. See a game in Table \ref{ga-1}. It is an example of the so-called Prisoners' Dilemma. Pure strategies of Player 1 and 2 are $X$ and $Y$. The left side number in each cell represents the payoff of Player 1 and the right side number represents the payoff of Player 2. Let $p_X$ and $1-p_X$ denote the probabilities that Player 1 chooses, respectively, $X$ and $Y$, and $q_X$ and $1-q_X$ denote the probabilities for Player 2. Denote the expected payoffs of Player 1 and 2 by $\pi_1(p_X, q_X)$ and $\pi_2(p_X, q_X)$. Then,
\begin{align*}
\pi_1(p_X, q_X)&=2p_Xq_X+3(1-p_X)q_X+(1-p_X)(1-q_X)\\
&=1-p_X+2q_X,
\end{align*}
and
\begin{align*}
\pi_2(p_X, q_X)&=2p_Xq_X+3p_X(1-q_X)+(1-p_X)(1-q_X)\\
&=1-q_X+2p_X.
\end{align*}
Denote the payoff of Player 1 when he chooses $X$ by $\pi_1(X, q_X)$, and that when he chooses $Y$ by $\pi_1(Y, q_X)$. Similarly for Player B. Then,
\[\pi_1(Y, q_X)=1+2q_X>\pi_1(p_X, q_X)\ \mathrm{for\ any}\ q_X\ \mathrm{and}\ p_X>0,\]
\[\pi_2(p_X,Y)=1+2p_X>\pi_2(p_X, q_X)\ \mathrm{for\ any}\ p_X\ \mathrm{and}\ q_X>0.\]

We can select points in $\mathbf{P}$ in (1) of Definition \ref{sln2} which correspond to vertices of $\Delta$ so that $\pi_1(Y, q_X)>\pi_1(p_X, q_X)$ and $\pi_2(p_X, Y)>\pi_2(p_X, q_X)$ hold.

Consider two sequences of $p_X$, $(p_X(m))_{m\geq 1}$ and $(p'_X(m))_{m\geq 1}$ such that $p_X(m)>0$ and $p'_X(m)>0$. If $\max(\max(\pi_1(X, q_X),\pi_1(Y, q_X))-\pi_1(p_X(m), q_X), 0)=\max(\pi_1(Y, q_X)-\pi_1(p_X(m), q_X), 0)\longrightarrow 0$ and $\max(\max(\pi_1(X, q_X),\pi_1(Y, q_X))-\pi_1(p'_X(m), q_X), 0)=\max(\pi_1(Y,q_X)-\pi_1(p'_X(m),q_X),0)\longrightarrow 0$, then $p_X(m)\longrightarrow 0$, $p'_X(m)\longrightarrow 0$ and $|p_X(m)-p'_X(m)|\longrightarrow 0$. 

Consider two sequences of $q_X$, $(q_X(m))_{m\geq 1}$ and $(q'_X(m))_{m\geq 1}$ such that $q_X(m)>0$ and $q'_X(m)>0$. If $\max(\max(\pi_2(p_X, X),\pi_2(p_X, Y))-\pi_2(p_X, q_X(m)), 0)=\max(\pi_2(p_X, Y)-\pi_2(p_X, q_X(m)), 0)\longrightarrow 0$ and $\max(\max(\pi_2(p_X, X),\pi_2(p_X, Y))-\pi_2(p_X, q'_X(m)), 0)=\max(\pi_2(p_X,Y)-\pi_2(p_X,q'_X(m)),0)\longrightarrow 0$, then $q_X(m)\longrightarrow 0$, $q'_X(m)\longrightarrow 0$ and $|q_X(m)-q'_X(m)|\longrightarrow 0$. 

Therefore, the payoff functions are sequentially locally non-constant.

%Thus, we can find a point $(p_X', q_X')$ in a neighborhood of any point $(p_X, q_X)$ satisfying $\pi_1(Y, q'_X)>\pi_1(p'_X, q'_X)$ or $\pi_2(p_X', Y)>\pi_2(p'_X, q'_X)$, and so the payoff functions of this game satisfy the sequential local non-constancy.

\begin{table}[htpb]
\begin{center}
\begin{tabular}{c|c|c|c|}
\multicolumn{4}{c}{\ \ \ \ \ \ \ \ \ \ \ \ \ \ \ \ Player 2} \\ \cline{2-4}
& &X&Y\\ \cline{2-4}
Player&X&2, 1 &0, 0\\ \cline{2-4}
1&Y& 0, 0& 1, 2\\ \cline{2-4}
\end{tabular}
\vspace{.3cm}
\caption{Example of game 2}
\label{ga-3}
\end{center}
\end{table}

Let us consider another example. See a game in Table \ref{ga-3}. It is an example of the so-called Battle of the Sexes Game. Notations are the same as those in the previous example. The expected payoffs of players are as follows;
\begin{align*}
\pi_1(p_X, q_X)&=2p_Xq_X+(1-p_X)(1-q_X)\\
&=1+p_X(3q_X-1)-q_X,
\end{align*}
\[\pi_1(X, q_X)=2q_X,\]
\[\pi_1(Y, q_X)=1-q_X,\]
\begin{align*}
\pi_2(p_X, q_X)&=p_Xq_X+2(1-p_X)(1-q_X)\\
&=2+q_X(3p_X-2)-2p_X,
\end{align*}
\[\pi_2(p_X, X)=p_X,\]
and
\[\pi_2(p_X, Y)=2-2p_X.\]
Then,
\begin{itemize}
\item When $q_X>\frac{1}{3}$, $\pi_1(X, q_X)>\pi_1(p_X, q_X)$ for $p_X<1$.
\item When $q_X<\frac{1}{3}$, $\pi_1(Y, q_X)>\pi_1(p_X, q_X)$ for $p_X>0$.
\item When $p_X>\frac{2}{3}$, $\pi_2(p_X, X)>\pi_2(p_X, q_X)$ for $q_X<1$ .
\item When $p_X<\frac{2}{3}$, $\pi_2(p_X, Y)>\pi_2(p_X, q_X)$ for $q_X>0$ .
\end{itemize}
We can select points in $\mathbf{P}$ in (1) of Definition \ref{sln2} which correspond to vertices of $\Delta$ so that $\pi_1(X, q_X)>\pi_1(p_X, q_X)$, $\pi_1(Y, q_X)>\pi_1(p_X, q_X)$, $\pi_2(p_X, X)>\pi_2(p_X, q_X)$ or $\pi_2(p_X, Y)>\pi_2(p_X, q_X)$ hold.

Consider sequences $(p_X(m))_{m\geq 1}$, $(p'_X(m))_{m\geq 1}$, $(q_X(m))_{m\geq 1}$ and $(q'_X(m))_{m\geq 1}$.
\begin{enumerate}
	\item When $p_X>\frac{2}{3}$, $q_X>\frac{1}{3}$, if $\max(\pi_1(X, q_X)-\pi_1(p_X(m), q_X),0)\longrightarrow 0$ and $\max(\pi_1(X, q_X)-\pi_1(p'_X(m), q_X),0)\longrightarrow 0$, then $p_X(m)\longrightarrow 1$, $p'_X(m)\longrightarrow 1$ and $|p_X(m)-p'_X(m)|\longrightarrow 0$.

 If $\max(\pi_2(p_X, X)-\pi_2(p_X, q_X(m)),0)\longrightarrow 0$ and $\max(\pi_2(p_X, X)-\pi_2(p_X, q'_X(m)),0)\longrightarrow 0$, then $q_X(m)\longrightarrow 1$, $q'_X(m)\longrightarrow 1$ and $|q_X(m)-q'_X(m)|\longrightarrow 0$.

	\item When $p_X<\frac{2}{3}$, $q_X<\frac{1}{3}$, if $\max(\pi_1(Y, q_X)-\pi_1(p_X(m), q_X),0)\longrightarrow 0$ and $\max(\pi_1(Y, q_X)-\pi_1(p'_X(m), q_X),0)\longrightarrow 0$, then $p_X(m)\longrightarrow 0$, $p'_X(m)\longrightarrow 0$ and $|p_X(m)-p'_X(m)|\longrightarrow 0$.

 If $\max(\pi_2(p_X, Y)-\pi_2(p_X, q_X(m)),0)\longrightarrow 0$ and $\max(\pi_2(p_X, Y)-\pi_2(p_X, q'_X(m)),0)\longrightarrow 0$, then $q_X(m)\longrightarrow 0$, $q'_X(m)\longrightarrow 0$ and $|q_X(m)-q'_X(m)|\longrightarrow 0$.

	\item When $p_X<\frac{2}{3}$, $q_X>\frac{1}{3}$, there exists no pair of sequences $(p_X(m))_{m\geq 1}$ and $(q_X(m))_{m\geq 1}$ such that $\max(\pi_1(X, q_X)-\pi_1(p_X(m), q_X),0)\longrightarrow 0$ and $\max(\pi_2(p_X, Y)-\pi_2(p_X, q_X(m)),0)\longrightarrow 0$.

	\item When $p_X>\frac{2}{3}$, $q_X<\frac{1}{3}$, there exists no pair of sequences $(p_X(m))_{m\geq 1}$ and $(q_X(m))_{m\geq 1}$ such that $\max(\pi_1(Y, q_X)-\pi_1(p_X(m), q_X),0)\longrightarrow 0$ and $\max(\pi_2(p_X, X)-\pi_2(p_X, q_X(m)),0)\longrightarrow 0$.

\item When $\frac{2}{3}-\varepsilon<p_X<\frac{2}{3}+\varepsilon$, $\frac{1}{3}-\varepsilon<q_X<\frac{1}{3}+\varepsilon$ with $0<\varepsilon<\frac{1}{3}$, if $\max(\pi_1(X, q_X)-\pi_1(p_X(m), q_X),0)\longrightarrow 0$, $\max(\pi_1(Y, q_X)-\pi_1(p_X(m), q_X),0)\longrightarrow 0$, $\max(\pi_2(p_X, X)-\pi_2(p_X, q_X(m)),0)\longrightarrow 0$ and $\max(\pi_2(p_X, Y)-\pi_2(p_X, q_X(m)),0)\longrightarrow 0$, then $(p_X(m), q_X(m))\longrightarrow (\frac{2}{3},\frac{1}{3})$ for all sequences $(p_X(m))_{m\geq 1}$ and $(q_X(m))_{m\geq 1}$.

The payoff functions are sequentially locally non-constant.
\end{enumerate}

\section{Concluding Remarks}

As a future research program we are studying the following themes.
\begin{enumerate}
	\item An application of the method of this paper to economic theory, in particular, the problem of the existence of an equilibrium in competitive economy with excess demand functions which have the property that is similar to sequential local non-constancy.
	\item A generalization of the result of this paper to Kakutani's fixed point theorem for multi-valued functions with property of sequential local non-constancy and its application to economic theory.
\end{enumerate}

\appendix
\section{Proof of Sperner's lemma}\label{app1}
We prove this lemma by induction about the dimension of $\Delta$. When $n=0$, we have only one point with the number 0. It is the unique 0-dimensional simplex. Therefore the lemma is trivial. When $n=1$, a partitioned 1-dimensional simplex is a segmented line. The endpoints of the line are labeled distinctly, by 0 and 1. Hence in moving from endpoint 0 to endpoint 1 the labeling must switch an odd number of times, that is, an odd number of edges labeled with 0 an 1 may be located in this way.

Next consider the case of 2 dimension. Assume that we have partitioned a 2-dimensional simplex (triangle) $\Delta$ as explained above. Consider the face of $\Delta$ labeled with 0 and 1\footnote{We call edges of triangle $\Delta$ \emph{faces} to distinguish between them and edges of a dual graph which we will consider later.}. It is the base of the triangle in Figure \ref{tria2}. Now we introduce a dual graph that has its nodes in each small triangle of $K$ plus one extra node outside the face of $\Delta$ labeled with 0 and 1 (putting a dot in each small triangle, and one dot outside $\Delta$). We define edges of the graph that connect two nodes if they share a side labeled with 0 and 1. See Figure \ref{tria2}. White circles are nodes of the graph, and thick lines are its edges. Since from the result of 1-dimensional case there are an odd number of faces of $K$ labeled with 0 and 1 contained in the face of $\Delta$ labeled with 0 and 1, there are an odd number of edges which connect the outside node and inside nodes. Thus, the outside node has odd degree. Since by the Handshaking lemma there are an even number of nodes which have odd degree, we have at least one node inside the triangle which has odd degree. Each node of our graph except for the outside node is contained in one of small triangles of $K$. Therefore, if a small triangle of $K$ has one face labeled with 0 and 1, the degree of the node in that triangle is 1; if a small triangle of $K$ has two such faces, the degree of the node in that triangle is 2, and if a small triangle of $K$ has no such face, the degree of the node in that triangle is 0. Thus, if the degree of a node is odd, it must be 1, and then the small triangle which contains this node is labeled with 0, 1 and 2 (fully labeled). In Figure \ref{tria2} triangles which contain one of the nodes $A$, $B$, $C$ are fully labeled triangles.

\begin{figure}[tpb]
\begin{center}
\includegraphics[height=9cm]{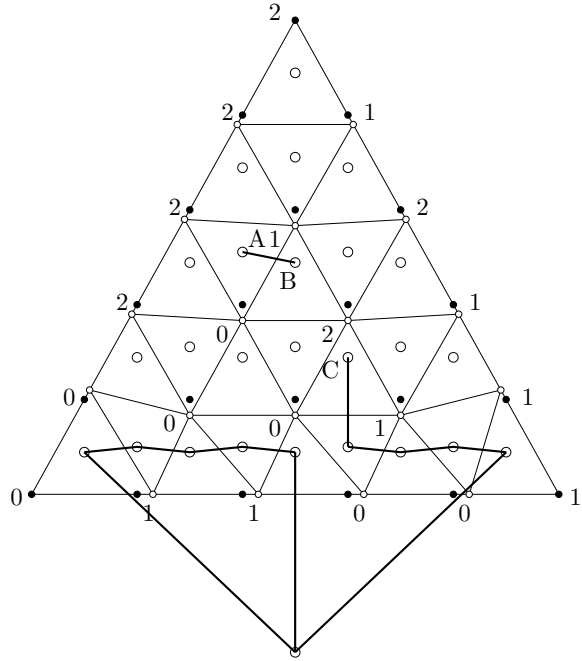}
\end{center}
	\vspace*{-0.13cm}
	\caption{Sperner's lemma}
	\label{tria2}
\end{figure}

Now assume that the lemma holds for dimensions up to $n-1$. Assume that we have partitioned an $n$-dimensional simplex $\Delta$. Consider the fully labeled face of $\Delta$ which is a fully labeled $n-1$-dimensional simplex. Again we introduce a dual graph that has its nodes in small $n$-dimensional simplices of $K$ plus one extra node outside the fully labeled face of $\Delta$ (putting a dot in each small $n$-dimensional simplex, and one dot outside $\Delta$). We define the edges of the graph that connect two nodes if they share a face labeled with 0, 1, $\dots$, $n-1$. Since from the result of $n-1$-dimensional case there are an odd number of fully labeled faces of small simplices of $K$ contained in the $n-1$-dimensional fully labeled face of $\Delta$, there are an odd number of edges which connect the outside node and inside nodes. Thus, the outside node has odd degree. Since, by the Handshaking lemma there are an even number of nodes which have odd degree, we have at least one node inside the simplex which has odd degree. Each node of our graph except for the outside node is contained in one of small $n$-dimensional simplices of $K$. Therefore, if a small simplex of $K$ has one fully labeled face, the degree of the node in that simplex is 1; if a small simplex of $K$ has two such faces, the degree of the node in that simplex is 2, and if a small simplex of $K$ has no such face, the degree of the node in that simplex is 0. Thus, if the degree of a node is odd, it must be 1, and then the small simplex which contains this node is fully labeled.

\begin{quote}
If the number (label) of a vertex other than vertices labeled with 0, 1, $\dots$, $n-1$ of an $n$-dimensional simplex which contains a fully labeled $n-1$-dimensional face is $n$, then this $n$-dimensional simplex has one such face, and it is a fully labeled $n$-dimensional simplex. On the other hand, if the number of that vertex is other than $n$, then the $n$-dimensional simplex has two such faces.
\end{quote}

We have completed the proof of Sperner's lemma.

Since $n$ and partition of $\Delta$ are finite, the number of small simplices constructed by partition is also finite. Thus, we can constructively find a fully labeled $n$-dimensional simplex of $K$ through finite steps.

\begin{description}
\item[Acknowledgement] This work was supported in part by the Ministry of Education, Science, Sports and Culture of Japan, Grant-in-Aid for Scientific Research (C), 20530165.
\end{description}

\bibliography{yasuhito}
\end{document}